\documentclass[12pt]{article}
\usepackage[margin=26mm]{geometry}

\usepackage{amsthm}

\usepackage{amsmath}
\usepackage{amssymb}
\usepackage{graphicx}
\usepackage{algorithm}
\usepackage{url}
\usepackage{epsfig}
\usepackage{color}
\usepackage{multirow}

\usepackage{algorithm}
\usepackage[noend]{algpseudocode}
\usepackage{cleveref}

\usepackage{mathptmx}       

\newcommand{\Z}{\mathbb{Z}}

\DeclareMathOperator\im{im}

\newcommand{\col}[2]{#1_{#2}}
\newcommand{\row}[2]{#1^{#2}}
\newcommand{\entry}[3]{#1^{#2}_{#3}}

\newcommand{\lowest}{\mathrm{pivot}}
\newcommand{\bm}{D}
\newcommand{\noglobal}{g}

\newcommand{\ignore}[1]{}

\newtheorem{theorem}{Theorem}
\newtheorem{lemma}[theorem]{Lemma}

\newtheorem{corollary}[theorem]{Corollary}

\begin{document}

\title{Clear and Compress: Computing Persistent Homology in Chunks}
\author{
Ulrich Bauer\footnote{Institute of Science and Technology Austria (IST Austria), Klosterneuburg, Austria} \and Michael Kerber\footnote{IST Austria; Stanford University, CA, USA; Max-Planck-Center for Visual Computing and Communication, Saarbr\"ucken, Germany} \and Jan Reininghaus\footnote{Institute of Science and Technology Austria (IST Austria), Klosterneuburg, Austria}
}
\date{}

\maketitle

\abstract{
We present a parallelizable algorithm for computing the persistent homology
of a filtered chain complex. 
Our approach differs from the commonly used reduction algorithm by first
computing persistence pairs within local chunks,
then simplifying the unpaired columns, and finally applying standard
reduction on the simplified matrix.
The approach generalizes a technique by G\"unther et al., which uses discrete Morse Theory to compute persistence;
we derive the same worst-case complexity bound in a more general context.
The algorithm employs several practical optimization techniques 
which are of independent interest.
Our sequential implementation of the algorithm is competitive with state-of-the-art methods,
and we improve the performance through parallelized computation.
}

\section{Introduction}
Persistent homology has developed from a theoretical idea to an entire research area within the
field of computational topology. One of its core features is its multi-scale approach to analyzing and quantifying
topological features in data. Recent examples of application areas are
shape classification~\cite{ccgmo-gromov}, topological denoising~\cite{blw-optimal}, or developmental biology~\cite{medusa-top}.

A second major feature of persistent homology is the existence of a simple yet efficient computation method:
The standard reduction algorithm as described in~\cite{elz-topological,zc-computing} computes the persistence
pairs by a simple sequence of column operations; Algorithm~\ref{alg:lr_persistence} gives a complete description 
in just $10$ lines of pseudo-code. The worst-case complexity is cubic in the input size of the complex,
but the practical behavior has been observed to be closer to linear on average. Various variants have been proposed in order to improve the theoretical bounds~\cite{mms-zigzag,ck-output} or the practical behavior~\cite{dmv-dualities}.

Our first contribution consists of two simple optimization techniques of the standard reduction algorithm, which we call
\emph{clearing} 
and \emph{compression}.
Both approaches exploit the special structure of a filtered chain complex in order to significantly reduce the number of operations on real-world instances. However, the two methods cannot be easily combined because they require the columns of the boundary matrix to be processed in different orders. 
   
Our second contribution is a novel algorithm
that incorporates both of the above optimization techniques,
and is also suitable for parallelization.
It proceeds in three steps: In the first step, the columns of the matrix are partitioned into consecutive \emph{chunks}. 
Each chunk is reduced independently, applying the clearing optimization mentioned above.
In this step, the algorithm finds at least
all persistence pairs with (index) persistence less than the size of the smallest chunk; 
let $g$ be the number
of columns not paired within this first step.
In the second step, the $g$ unpaired columns
are compressed using the method mentioned above. After compression, each column has at most $g$ non-zero entries
and the unpaired columns form a nested $(g\times g)$-matrix.
In the third and final step, this nested matrix is reduced, again applying the clearing optimization.

The chunk algorithm is closely related to two other methods for computing persistence.
First of all, the \emph{spectral sequence algorithm}~\cite[\textsection VII.4]{eh-computational} decomposes the matrix into
blocks and proceeds in several phases, computing in phase $r$ the persistence pairs lying $r-1$ blocks apart. The first step of the 
chunk algorithm is equivalent to applying the first two phases of the spectral sequence approach.
Furthermore, the three step chunk algorithm is inspired 
by the approach of G\"unther et al.~\cite{Guenther_PLD2},
which combines persistence computation and discrete Morse theory for 3D image data. 
The first step of that algorithm consists in constructing a discrete gradient field consistent with the input function; such a gradient field can be interpreted as a set of persistence pairs that are incident in the complex and have persistence $0$. We replace this method by local persistence computations,
 allowing us to find pairs with are not incident in the complex. 

We analyze the chunk algorithm in terms of time complexity. Let $n$ be the number of generators (simplices, cells)
of the chain complex, $m$ the number of chunks, $\ell$ the maximal size of a chunk, and $g$ as above.
We obtain a worst-case bound of
$$O(m\ell^3+g\ell n + g^3),$$
where the three terms reflect the worst-case running times of the three steps.
For the filtration of a cubical complex induced by a $d$-dimensional grayscale image (with $d$ some fixed constant),
this bound simplifies to
$$O(gn+g^3),$$
if the chunks are given by the cells appearing simultaneously in the filtration.
This bound improves on the previous general bound of $O(g^2n\log n)$ from~\cite{ck-persistent}.
Moreover, it matches the bound in~\cite{Guenther_PLD2}, but applies to arbitrary dimensions.
Of course, the bound is still cubic if expressed only in terms of $n$ because $g\in O(n)$ in the worst case.

We implemented a sequential and a parallelized version of the chunk algorithm;
both are publicly available in our new PHAT library (\url{http://phat.googlecode.com/}).
The sequential code already outperforms the standard reduction algorithm and is competitive
to other variants with a good practical behavior.
The parallelized version using 12 cores yields a speed-up factor between 3 and 11 (depending on the example) in our tests, making the implementation the fastest among the considered choices.
This is the first result where the usefulness of parallelization is shown for the problem
of persistence computation through practical experiments.

\section{Background}
\label{sec:background}

This section summarizes the theoretical foundations of persistent homology
as needed in this work. 
We limit our scope to simplicial homology over $\Z_2$ just for the sake of simplicity in the description; our methods generalize to chain complexes 
over arbitrary fields.

\paragraph{Homology}
Homology is an algebraic tool for analyzing the connectivity of topological spaces.
Let $K$ be a simplicial complex of dimension $d$. 
In any dimension $p$, we call a \emph{$p$-chain} a formal sum
of the $p$-simplices of $K$ with $\Z_2$ coefficients. The $p$-chains
form a group called the \emph{$p$th chain group} $C_p$.
The \emph{boundary} of a $p$-simplex $\sigma$ is the $(p-1)$-chain formed
by the sum of all faces of $\sigma$ of codimension $1$. This operation extends linearly
to a \emph{boundary operator} $\delta:C_p\rightarrow C_{p-1}$.
A $p$-chain $\gamma$ is a \emph{$p$-cycle} if $\delta(\gamma)=0$.
The $p$-cycles form a subgroup of the $p$-chains, which we call
the \emph{$p$th cycle group} $Z_p$. A $p$-chain $\gamma$ is called a 
\emph{$p$-boundary} if $\gamma=\delta(\xi)$ for some $(p+1)$-chain $\xi$.
Again, the $p$-boundaries form a group $B_p$, and since $\delta(\delta(\xi))=0$
for any chain $\xi$, $p$-boundaries are $p$-cycles, 
and so $B_p$ is a subgroup of $Z_p$.
The \emph{$p$th homology group} $H_p$ is defined as the quotient group
$Z_p/B_p$. The rank of $H_p$ is denoted by $\beta_p$ and is called
the \emph{$p$th Betti number}.
In our case of $\Z_2$ coefficients, the homology group is a vector space
isomorphic to $\Z_2^{\beta_p}$, hence it is completely determined
by the Betti number.
Roughly speaking, the Betti numbers in dimension $0$, $1$, and $2$ yield the number
of connected components, tunnels, and voids of $K$, respectively.

\paragraph{Persistence}
Let $\{\sigma_1,\ldots,\sigma_n\}$ denote the simplices of $K$.
We assume that for each $i\leq n$, 
$K_i:=\{\sigma_1,\ldots,\sigma_i\}$ is a simplicial complex again.
The sequence of inclusions 
$\emptyset=K_0\subset\ldots\subset K_i\ldots\subset K_n=K$ is called
a \emph{simplexwise filtration} of $K$.
For every dimension $p$ and every $K_i$, we have a homology group $H_p(K_i)$;
we usually consider all dimensions at once and write $H(K_i)$
for the direct sum of the homology groups of $K_i$ in all dimensions.
The inclusion $K_{i}\hookrightarrow K_{i+1}$ induces a homomorphism
$g_{i}^{i+1}: H(K_i)\rightarrow H(K_{i+1})$ on the homology groups. 
These homomorphisms compose and we can define $g_{i}^{j}:H(K_i)\rightarrow H(K_j)$
for any $i\leq j$. We say that a class $\alpha\in H(K_\ell)$ is
\emph{born at (index) $i$} if $\alpha\in\im g_i^\ell$ 
but $\alpha\notin\im g_{i-1}^\ell$. 
A class $\alpha$ born at index $i$ \emph{dies entering (index)~$j$} if
$g_i^j(\alpha)\in\im g_{i-1}^j$ but 
$g_i^{j-1}(\alpha)\notin\im g_{i-1}^{j-1}$.
In this case, the index pair $(i,j)$ is called a \emph{persistence pair}, and the difference $j-i$ is the \emph{(index) persistence}
of the pair. 
The transition from $K_{i-1}$ to $K_i$ either causes the birth or the death
of an homology class. We call the added simplex $\sigma_i$ \emph{positive}
if it causes a birth and \emph{negative} if it causes a death. 
Note that homology classes of the full complex $K$ do not die during
the filtration. We call a simplex $\sigma_i$ that gives birth to 
such a class \emph{essential}. All other simplices are called 
\emph{inessential}.

\paragraph{Boundary matrix}
For a matrix $M\in\Z_2^{n\times n}$, we let $\col{M}{j}$
denote its $j$-th column, $\row{M}{i}$
its $i$-th row, and $\entry{M}{i}{j}\in\Z_2$
its entry in row $i$ and column $j$.
For a non-zero column $0\neq \col{M}{j}=(m_1,\ldots,m_n)\in\Z^n_2$, 
we set $\lowest(\col{M}{j}):=\max\{i=1,\ldots,n\mid m_i=1\}$
and call it the \emph{pivot index} of that column.

The \emph{boundary matrix} $\bm\in(\Z_2)^{n\times n}$ 
of a simplexwise filtration 
$(K_i)_i$ is a $n\times n$ matrix with
$\entry{\bm}{i}{j}=1$ 
if and only if $\sigma_i$ is a face of $\sigma_j$ of codimension~$1$.
In other words, the $j$th column of $\bm$ encodes the boundary
of $\sigma_j$. 
$\bm$ is an upper-triangular matrix because any face of $\sigma_j$
must precede $\sigma_j$ in the filtration.
Since the $j$th row and column of $\bm$ corresponds to the $j$th
simplex $\sigma_j$ of the filtration, we can talk about
\emph{positive columns}, \emph{negative columns}, and \emph{essential columns}
in a natural way, and similarly for rows.

\paragraph{The reduction algorithm}
A column operation of the form
$\col{M}{j}\gets \col{M}{j}+\col{M}{k}$ is called \emph{left-to-right} if $k<j$.
We call a matrix $M'$ \emph{derived from $M$}
if $M$ can be transformed
into $M'$ by left-to-right operations.
Note that in a derivation $M'$ of $M$, the $j$th column can be expressed
as a linear combination of the columns $1,\ldots,j$ of $M$, and this
linear combination includes $\col{M}{j}$.
We call a matrix $R$ \emph{reduced} if no two non-zero columns have the same 
pivot index. If $R$ is derived from $M$, we call it a \emph{reduction of $M$}.
In this case, we define
\begin{eqnarray*}
P_R&:=&\{(i,j)\mid \col{R}{j}\neq 0 \wedge i=\lowest(\col{R}{j})\}\\
E_R&:=&\{i\mid \col{R}{i}=0\wedge \lowest(\col{R}{j})\neq i\forall j=1,\ldots,n\}.
\end{eqnarray*}
Although the reduction matrix $R$ is not unique, 
the sets $P_R$ and $E_R$ are the same for any choice of reduction;
therefore, we can define
$P_{M}$ and $E_{M}$ to be equal to $P_R$ and $E_R$ for any reduction $R$
of $M$. We call the set $P$ the \emph{persistence pairs} of $M$.
When obvious from the context, we omit the subscripts
and simply write $P$ for the persistence pairs.
For the boundary matrix $\bm$ of $K$, 
the pairs $(i,j)\in P$ are the persistence pairs of the filtration $(K_i)_{0\leq i\leq n}$,
and the indices in $E$ correspond to the essential simplices of the complex.
Note that $E$ is uniquely determined by~$P$ and~$n$ as the indices between $1$ and $n$ that do not appear in any pair of~$P$. 

The simplest way of reducing $\bm$ is to process columns from left 
to right; for every column, other columns are added from the left
until the pivot index is unique
(Algorithm~\ref{alg:lr_persistence}).
A lookup table can be used
to identify the next column to be added in constant time. 
A flag is used
for every column denoting
whether a persistence pair with the column index has already been found.
After termination, the unpaired columns correspond to the essential columns.
The running time is at most cubic in $n$,
and this bound is actually tight for certain input filtrations,
as demonstrated in~\cite{morozov-persistence}. 

\begin{algorithm}[ht]
\caption{Left-to-right persistence computation}
\label{alg:lr_persistence}
\begin{algorithmic}[1]
\Procedure {Persistence\_left\_right}{$\bm$}
  \State $R\gets \bm$; $L\gets[0,\ldots,0]$; $P\gets\emptyset$ \Comment{$L\in\Z^n$}
  \For{$j=1,\ldots,n$}
  \While{$\col{R}{j}\neq 0$ and $L[\lowest(\col{R}{j})]\neq 0$}
  \State $\col{R}{j}\gets \col{R}{j}+\col{R}{L[\lowest(\col{R}{j})]}$
  \EndWhile
  \If{$\col{R}{j}\neq 0$} 
    \State $i\gets \lowest(\col{R}{j})$
    \State $L[i]\gets j$
    \State Mark columns $i$ and $j$ as paired and add $(i,j)$ to $P$
  \EndIf
  \EndFor
  \State {\bf return} $P$
\EndProcedure
\end{algorithmic}
\end{algorithm}

Let $M$ be derived from $\bm$. A column $\col{M}{j}$ of $M$ is called
\emph{reduced} if either it is zero, 
or if $(i,j)\in P$ with $i=\lowest(\col{M}{j})$.
With this definition, a matrix~$M$ is a reduction of $\bm$
if and only if every column is reduced.

\vspace{-.01\textheight}
\section{Speed-ups}
\label{sec:twists}
\vspace{-.01\textheight}

Algorithm~\ref{alg:lr_persistence} describes the simplest way
of reducing the boundary matrix, but it performs more operations
than actually necessary to compute the persistence pairs.
We now present two simple techniques which both lead to
a significant decrease in the number of required operations.

\paragraph{Clearing positive columns}
The key insight behind our first optimization is the following fact: if $i$ appears as the pivot in a reduced column of $M$, the index~$i$ is positive and hence there exists
a sequence of left-to-right operations on $\col{M}{i}$ that turn it to zero. 
Instead of explicitly executing this sequence of operations, 
we define the \emph{clear} operation 
by setting column $\col{M}{i}$ to zero directly.
Informally speaking, a clear is a shortcut to avoid some column operations
in the reduction when the result is evident.

In order to apply this optimization, we change the traversal order
in the reduction by first reducing the columns corresponding to simplices with dimension $d$
(from left to right), then all columns with dimension $d-1$, and so on.
After having reduced all columns with dimension $\delta$, we have
found all positive inessential columns with dimension $\delta-1$
and clear them before continuing with $\delta-1$.
This way all positive inessential columns of the complex are cleared
without performing any column additions on them. See~\cite{ck-persistent}
for a more detailed description. 

\paragraph{Compression}
Alternatively, we can try to save arithmetic operations by reducing the number of non-zero rows
among the unpaired columns. A useful observation in this context is given next.
\begin{lemma}\label{lem:clearing}
Let $\col{M}{j}$
be a non-zero column of $M$ with $i=\lowest(\col{M}{j})$.
Then $\col{M}{i}$ is a positive and inessential column.
\end{lemma}
\begin{proof}
The statement is clearly true if $\col{M}{j}$ is reduced, because in this case $(i,j)$ is a persistence pair.
If $\col{M}{j}$ is not reduced, this means that after applying some sequence of left-to-right column operations, some reduced column has $i$ as pivot index.
\end{proof}
\begin{corollary}
Let
$\col{M}{i}$ be a negative column of $M$.
Then $i$ is not the pivot index of any column in~$M$.
\end{corollary}
As a consequence, whenever a negative column with index~$j$
has been reduced, row~$j$ can be set to zero before further reducing. 
\begin{corollary}\label{lem:dual_twist}
Let $\col{M}{i}$ be a negative column and let $\col{M}{j}$ be a column with $\entry{M}{i}{j}=1$. Then setting $\entry{M}{i}{j}$ to zero does not affect the pairs.
\end{corollary}

We can even do more: let $i$ be the pivot index of the reduced column $\col{M}{j}$ and assume that the submatrix of $M$ with column indices $\{1,\dots ,j\}$ and row indices $\{i,\dots ,n\}$ is reduced, i.e., the pivot indices are unique in this submatrix.
By adding column~$j$ to each unreduced column in the matrix
that has a non-zero entry at row~$i$, we can eliminate all non-zero entries
in row $i$ from the unreduced columns. Note that if $k<j$ and $\entry{M}{i}{k}\neq0$, then  $\lowest(\col{M}{k})\geq i$ and thus, by assumption, $\col{M}{k}$ must be a reduced negative column.
Therefore, for each unreduced column $\col{M}{k}$, 
the operation $\col{M}{k}\gets \col{M}{k}+\col{M}{j}$ is a left-to-right addition and thus does not affect the pairs.

\section{Reduction in chunks}
\label{sec:chunks}

The two optimization techniques from Section~\ref{sec:twists} both yield
significant speed-ups, but they are not easily combinable, because
clearing requires to process a simplex before its faces, whereas
compression works in the opposite direction. 
In this section, we present an algorithm which combines
both optimization techniques. 

Let $m\in\mathbb N$. Fix $m+1$ numbers
$0=t_0<t_1<\ldots<t_{m-1}<t_m=n$ and define the $i$th \emph{chunk}
of~$\bm$ 
to be the columns of $\bm$ with indices 
$\{t_{i-1}+1,\ldots,t_{i}\}$. We call a column~$\col{\bm}{j}$ \emph{local}
if it forms a persistence pair with another column in the same chunk or in one of the adjacent chunks. 
In this case, we also call the persistence pair local.
Non-local columns (and pairs) are called \emph{global}. 
If $\ell$ is a lower bound on the size of each chuck, then
every global persistence pair has index persistence at least~$\ell$.
We also call an index $j$ local if the $j$th column of $\bm$
is local, and the same for global.
We denote the number of global columns in $\bm$ by $\noglobal$.
The high-level description of our new algorithm consists of three steps:
\begin{enumerate}
\item Partially reduce every chunk independently, applying the clearing optimization, so that all local columns are completely reduced. 
\item Independently compress every global column such that 
      all its non-zero entries are global.
\item Reduce the submatrix consisting only of the global rows and columns.
\end{enumerate}
We  give details about the three steps in the rest of this section.
The first two steps can be 
performed in parallel, whereas the third step only needs to reduce
a matrix of size $\noglobal\times\noglobal$ instead of $n\times n$. In many situations,
$\noglobal$ is significantly smaller than $n$.

\paragraph{Local chunk reduction}
The first step of our algorithm computes the local pairs
by performing two phases of the spectral sequence algorithm~\cite[\textsection VII.4]{eh-computational}. 
Concretely, we apply left-to-right operations as usual,
but in the first phase we only add columns from the same chunk,
and in the second phase we only add columns from both the same chunk and its left neighbor.
After phase $r$, for each $b\in\{r,\dots,m\}$ the submatrix with column indices $\{1,\dots ,t_b\}$ and row indices $\{t_{b-2}+1,\dots ,n\}$ is reduced.
If the reduction of column~$j$ stops at a pivot index $i>t_{b-r}$, row $j$ cannot be reduced any further by adding any column, so
we identify $(i,j)$ as a local persistence pair. 
Conversely, any local pair $(i,j)$ is detected by this method after two phases.
We incorporate the clearing operation for efficiency, that is, we proceed
in decreasing dimension and set detected local positive columns to zero;
see Algorithm~\ref{alg:local_chunk}.
After its execution, 
$L[i]$ contains the index of the 
local negative column with pivot index~$i$ for any local positive column~$i$,
and the resulting matrix~$R$ is a derivation
of $\bm$ in which all local columns are reduced.

\begin{algorithm}[ht]
\caption{Local chunk reduction}
\label{alg:local_chunk}
\begin{algorithmic}[1]
\Procedure {Local\_reduction}{$M$, $t_0,\ldots,t_m$}
  \State $R\gets M$; $L\gets[0,\ldots,0]$; $P\gets\emptyset$ \Comment{$L\in\Z^n$}
  \For{$\delta=d,\ldots,0$}
  \For{$r=1,2$}
  \Comment{Perform two phases of the spectral sequence algorithm}
   \For{$b=r,\ldots,m$}
   \Comment{Loop is parallelizable}
    \For{$j=t_{b-1}+1,\ldots,t_{b}$ with $\dim\sigma_j = \delta$}
      \If{$j$ is not marked as paired}
        \While{$\col{R}{j}\neq 0\wedge L[\lowest(\col{R}{j})]\neq 0 \wedge \lowest(\col{R}{j})>t_{b-r}$}
          \State $\col{R}{j}\gets \col{R}{j}+\col{R}{L[\lowest(j)]}$
        \EndWhile
        \If{$\col{R}{j}\neq 0$}
          \State $i\gets\lowest(\col{R}{j})$
          \If{$i> t_{b-r}$}
            \State $L[i]\gets j$
            \State $\col{R}{j}\gets 0$\Comment{Clear column $i$}
            \State \label{alg:mark_paired}Mark $i$ and $j$ as paired and add $(i,j)$ to $P$
          \EndIf
        \EndIf
      \EndIf
    \EndFor
   \EndFor
  \EndFor
  \EndFor
  \State \textbf{return} $(R,L,P)$
\EndProcedure
\end{algorithmic}
\end{algorithm}

\begin{algorithm}[ht]
\caption{Determining active entries}
\label{alg:active_columns}
\begin{algorithmic}[1]

\Procedure {Mark\_active\_entries}{$R$}
  \For{each unpaired column $k$}
  \Comment{Loop is parallelizable}
  \State\Call{Mark\_column}{$R$, $k$}
  \EndFor
\EndProcedure

\Function {Mark\_column}{$R$, $k$}
  \If {$k$ is marked as active/inactive} \Return true/false \EndIf
  \For{each non-zero entry index $i$ of $\col{R}{k}$}
    \If{$\ell$ is unpaired}
      \State mark $k$ as active and \Return true 
    \ElsIf{$i$ is positive}
    \State $j \gets L[i]$ \Comment{$(i,j)$ is persistence pair}
    \If{ $j\neq k$ and \Call{Mark\_column}{$R$, $j$}}
    \State mark $k$ as active and \Return true 
  \EndIf
    \EndIf
  \EndFor
  \State mark $k$ as inactive and \Return false
\EndFunction

\end{algorithmic}
\end{algorithm}

\paragraph{Global column compression}
Let $R$ be the matrix returned by Algorithm~\ref{alg:local_chunk}.
Before computing the global persistence pairs,
we first compress the global columns, using the ideas from
Section~\ref{sec:twists}; recall that negative rows
can simply be set to zero, while entries in positive rows
can be eliminated by an appropriate column addition.
Note, however, that a full column addition might actually be unnecessary:
for instance, if all non-zero row indices in the added column belong to negative columns
(except for the pivot), the entry in the local positive row 
could just have been zeroed out in the same way 
as in Corollary~\ref{lem:dual_twist}.
Speaking more generally, it is more efficient to avoid column additions
that have no consequences for global indices, neither directly nor indirectly.

In the spirit of this observation, we call an index $i$
\emph{inactive} if either it is a local negative index 
or if $(i,j)$ is a local pair and all indices of non-zero entries in column $\col{R}{j}$
apart from $i$ are inactive. Otherwise, the index is called \emph{active}.
By induction and Corollary~\ref{lem:dual_twist}, we can show:
\begin{lemma}\label{lem:active}
Let $i$ be an inactive index 
and let $\col{M}{j}$ be any column 
with $\entry{M}{i}{j}=1$.  
Then setting $\entry{M}{i}{j}$ to zero does not affect the persistence pairs.
\end{lemma}
The compression proceeds in two steps: first, every non-zero entry of a global column 
is classified as active or inactive (using depth-first search; see Algorithm~\ref{alg:active_columns}). Then, we iterate over the
global columns, set all entries with inactive index to zero,
and eliminate any non-zero entry with a local positive index $\ell$ by column addition with $L[\ell]$ (see Algorithm~\ref{alg:compress}).
After this process, we obtain a matrix $R'$ 
with the same persistence pairs as $R$, such that the global columns of $R'$ 
have non-zero entries only in the global rows.

\begin{algorithm}[ht]
\caption{Global column compression}
\label{alg:compress}

\begin{algorithmic}[1]
\Procedure {Compress}{$R,k$}
  \State \textbf{Uses variables: } $L$
  \For{each non-zero entry index $\ell$ of $\col{R}{k}$ in decreasing order}
      \If{$\ell$ is paired}
      \If{$\ell$ is inactive}
         \State $\entry{R}{k}{\ell}\gets 0$
      \Else
         \State $j\gets L[\ell]$ \Comment{$(\ell,j)$ is persistence pair}
         \State $\col{R}{k}\gets \col{R}{k}+\col{R}{j}$
      \EndIf
      \EndIf
  \EndFor
\EndProcedure
\end{algorithmic}
\end{algorithm}

\paragraph{Submatrix reduction}
After having compressed all global columns, 
these form a $\noglobal\times\noglobal$ matrix ``nested'' in
$R$ (recall that $\noglobal$ is the number of global columns). To complete
the computation of the persistence pairs, we simply
perform standard reduction on the remaining matrix.
For efficiency, we perform steps 2 and 3 alternatingly for all dimensions
in decreasing order and apply the clearing optimization; 
this way, we avoid the compression of positive global columns.
Algorithm~\ref{alg:persistence_chunk} summarizes the whole method.

\begin{algorithm}[ht]
\caption{Persistence in chunks}
\label{alg:persistence_chunk}
\begin{algorithmic}[1]
\Procedure {Persistence\_in\_chunks}{$\bm$, $t_0,\ldots,t_m$}
  \State $(R,L,P)\gets$\Call{Local\_reduction}{$\bm$, $t_0,\ldots,t_m$}
  \Comment{step 1: reduce local columns}

  \State \Call{Mark\_active\_entries}{$R$}
  \For{$\delta=d,\ldots,0$}
  \State \Comment{step 2: compress global columns}
  \For{$j=1,\ldots,n$  with $\dim\sigma_j = \delta$}
  \Comment{Loop is parallelizable}

  \If{column $j$ is not paired}

    \State \Call{Compress}{$R,j$}
  \EndIf
  \EndFor
  \For{$j=1,\ldots,n$ with $\dim\sigma_j = \delta$}
    \Comment{step 3: reduce global columns}
      \While{$\col{R}{j}\neq 0\wedge L[\lowest(\col{R}{j})]\neq 0$}
        \State $\col{R}{j}\gets \col{R}{j}+\col{R}{L[\lowest(j)]}$
      \EndWhile
  \If{$\col{R}{j}\neq 0$}
  \State $i\gets\lowest(\col{R}{j})$
  \State $L[i]\gets j$
  \State $\col{R}{i}\gets 0$\Comment{Clear column $i$}
  \State Mark $i$ and $j$ as paired and add $(i,j)$ to $P$
  \EndIf
  \EndFor
  \EndFor
  \State {\bf return} $P$
\EndProcedure

\end{algorithmic}
\end{algorithm}

\section{Analysis}
Algorithm~\ref{alg:persistence_chunk} permits a complexity analysis
depending on the following parameters: 
$n$, the number of simplices;
$m$, the number of chunks; 
$\ell$, the maximal size of a chunk;
and $\noglobal$, the number of global columns.
We assume that for any simplex, the number of non-zero faces of codimension $1$ is bounded by a constant (this is equivalent to assuming that the dimension of the complex is a constant).

\paragraph{General complexity}
We show that the complexity of Algorithm~\ref{alg:persistence_chunk} is
bounded by
\begin{equation}
O(m\ell^3 + \noglobal\ell n + \noglobal^3).
\label{eqn:complexity}
\end{equation}
The three summands correspond to the running times of the three
steps\footnote{The running time of the
third step could be lowered to $\noglobal^\omega$, where $\omega$ is the matrix-%
multiplication exponent, using the method of~\cite{mms-zigzag}.}.
Note that $\noglobal\in O(n)$ in the worst case.

For the complexity of Algorithm~\ref{alg:local_chunk}, we consider
the complexity of reducing one chunk, which consists of up to $\ell$ columns.
Within the local chunk reduction, every column is only added with
columns of the same or the previous chunk, so there are only up to $2\ell$ column additions
per column. Moreover, since the number of non-zero entries per column
in $\bm$ is assumed to be constant, there are only $O(\ell)$ many
entries that can possibly become non-zero during the local chunk reduction.
It follows that the local chunk reduction can be considered as a reduction
on a matrix with $\ell$ columns and $O(\ell)$ rows. If we represent 
columns by linked lists (containing the non-zero indices in sorted order), one column operation can be done in $O(\ell)$
primitive operations, which leads to a total complexity of $O(\ell^3)$
per chunk.

The computation of active columns in Algorithm~\ref{alg:active_columns} is done by depth-first search on a graph whose vertices are given by the columns and whose edges correspond to their non-zero entries. The number of edges is $O(n\ell)$, so we obtain a running time of $O(n\ell)$.

Next, we consider the cost of compressing a global column
with index $j$. After the previous step, the column has at most $O(\ell)$
non-zero entries. We transform the presentation of the column
from a linked list into a bit vector of size $n$.
In this representation, adding another column in list representation with $v$ entries to column~$j$ takes
time proportional to $v$.
In the worst case, we need to add all columns
with indices $1,\ldots,j-1$ to $j$. Each such column has $O(\ell)$ entries.
At the end, we transform the bit vector back into a linked list representation.
The total cost is $O(n+(j-1)\ell+n)=O(n\ell)$ per global
column.

Finally, the complexity of the global reduction is $O(\noglobal^3)$, as in the standard reduction.

\paragraph{Choosing chunks}
We discuss different choices of chunk size and their complexities.
A generic choice for an arbitrary complex is to choose $O(\sqrt{n})$ chunks
of size $O(\sqrt{n})$ each. With that, the complexity of \eqref{eqn:complexity}
becomes
$$O(n^2+\noglobal_1n\sqrt{n} + \noglobal_1^3).$$
Alternatively, choosing $O(\frac{n}{\log n})$ 
chunks of size $O(\log n)$, the complexity becomes
$$O(n\log^2 n + \noglobal_2n\log n+\noglobal_2^3).$$
We replaced $\noglobal$ by $\noglobal_1$ and $\noglobal_2$ to express that the number of global 
columns is different in both variants. In general, choosing larger chunks
is likely to produce less global columns, since
every global persistence pair has index persistence at least~$\ell$ (the size of the smallest chunk) .

\paragraph{Cubical complexes}
We consider an important special case of boundary matrices: consider
a $d$-di\-men\-sion\-al image with $p$ hypercubes, where each vertex
contains a grayscale value. We assume that the cubes
are triangulated conformally in order to get simplicial input~-- the argument also
works, however, for the case of cubical cells. 
We assign function values inductively, 
assigning to each simplex the maximal value of its faces.
Assuming that
all vertex values are distinct, the \emph{lower star}
of vertex $v$ is the set of all simplices which have the same function value as $v$. 
Filtering the simplices in a way that respects the order of the function values, 
we get a \emph{lower star filtration} of the image. 
Now choose the lower stars as the chunks in our reduction algorithm. Note that
the lower star is a subset of the star of the corresponding vertex,
which is of constant size (assuming that the dimension $d$ is constant).
Therefore, the complexity bound~\eqref{eqn:complexity} reduces to
$$O(n+\noglobal n+\noglobal^3)=O(\noglobal n+\noglobal^3).$$
Note that global columns with large index persistence
might still have very small, or even zero, persistence with respect to 
the function values, for instance in the presence of a flat region in the image
where many vertices have similar values.

\section{Experiments}
\label{sec:experiments}
We implemented two versions of the algorithm presented in Section~\ref{sec:chunks}: a sequential and a parallel version
(using \textsc{OpenMP}), in which the first two steps of the algorithm are performed simultaneously on each chunk and on each global column, respectively. In both cases, we use $\lfloor\sqrt{n}\rfloor$ as the chunk size. For a fair comparison, we also re-implemented the algorithms introduced in \cite{ck-persistent,elz-topological} in the same framework, that means, using the same data representations and low-level operations such as column additions. 
Our implementation is publicly available in our new \emph{PHAT} library
for computing persistence homology, available at \url{http://phat.googlecode.com/}.
Additionally, we compare to the memory efficient algorithm~\cite{Guenther_PLD2} based on discrete Morse theory~\cite{Forman1998b} and to the implementation of the persistent cohomology algorithm~\cite{dmv-dualities} found in the \textsc{Dionysus} library \cite{morozov2010dionysus}. 

To find out how these algorithms behave in practice, we apply them to five representative data sets. The first three are 3D image data sets with a resolution of $128^3$. The first of these is given by a Fourier sum with random coefficients and is representative of smooth data. The second is uniform noise. The third is the sum of the first two and represents large-scale structures with some small-scale noise. These data sets are illustrated in Figure~\ref{fig:data_sets} by an isosurface. 

\begin{figure}[htb]%
\centering%
a)\includegraphics[width=0.29\linewidth]{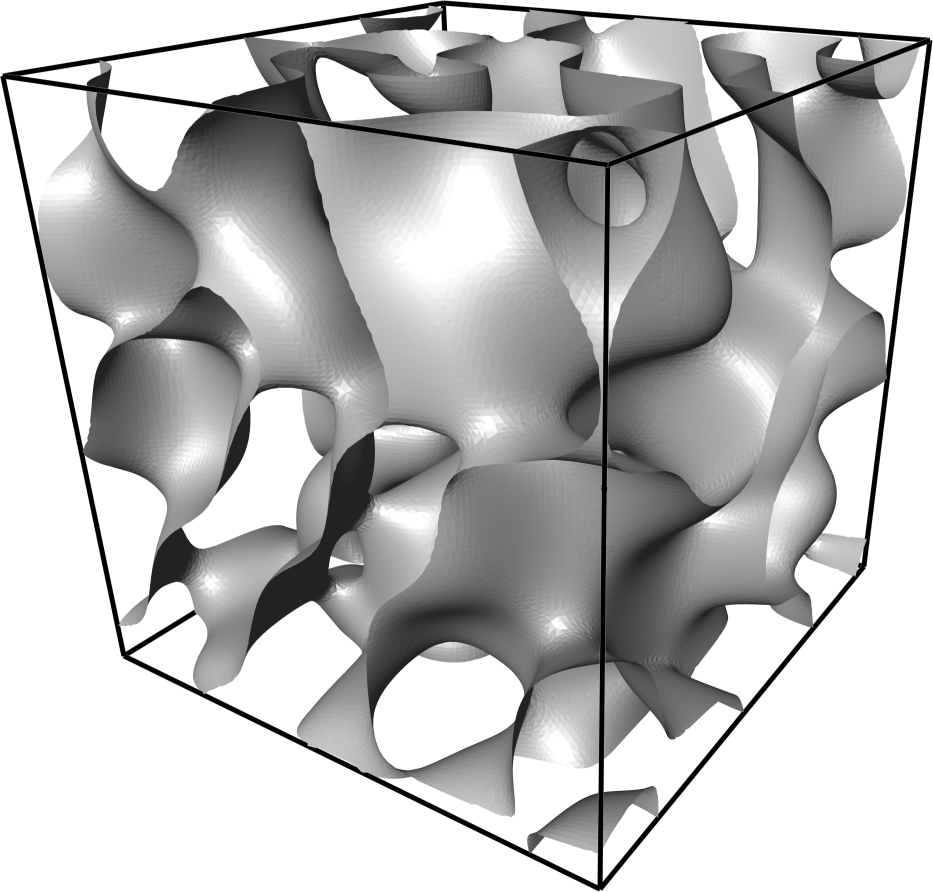}\quad
b)\includegraphics[width=0.29\linewidth]{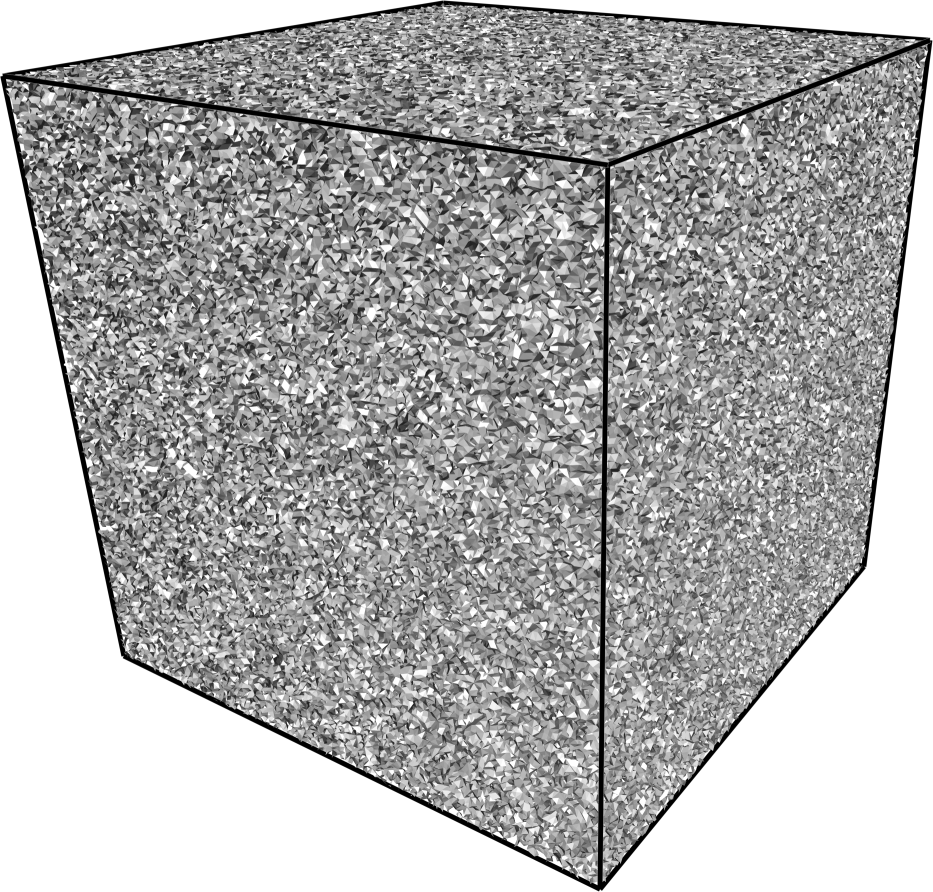}\quad
c)\includegraphics[width=0.29\linewidth]{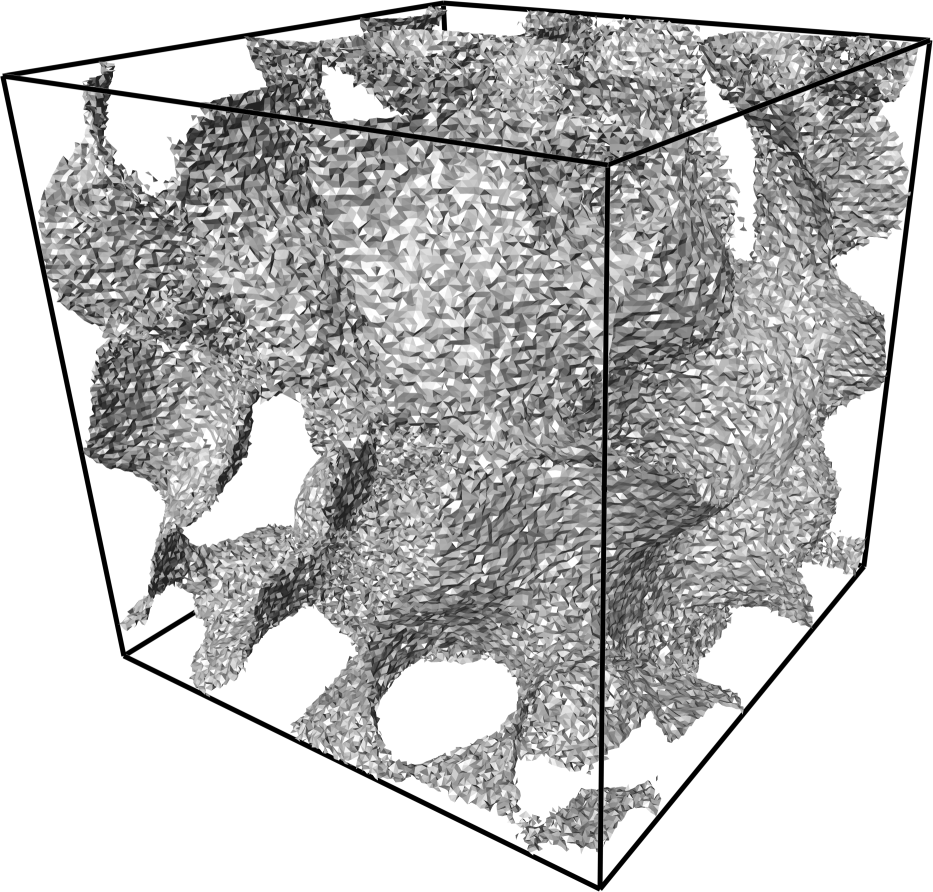}
\caption{A single isosurface of the representative data sets used to benchmark and compare our algorithm: a) a smooth data set, b) uniformly distributed noise, c) the sum of a) and b).}%
\label{fig:data_sets}%
\end{figure}

In addition to the lower star filtrations of these image data sets, we also consider an alpha shape filtration defined by 10000 samples of a torus embedded in
$\mathbb R^3$, and the 4-skeleton of the Rips filtration given by 50 points randomly chosen from the Mumford data set \cite{Mumford2003Images}.

As pointed out in \cite{dmv-dualities}, the pairs of persistent cohomology are the same as those of persistent homology. We therefore also applied all algorithms to the corresponding cochain filtration, given in matrix form by transposing the boundary matrix and reversing the order of the columns; this operation is denoted by $(\cdot)^{\perp}$. When reducing such a coboundary matrix with the clearing optimization, columns are processed in order of increasing instead of decreasing dimension.

Table~\ref{tab:running_times} contains the running times of the above algorithms applied to filtrations of these five data sets run on a PC with two Intel Xeon E5645 CPUs. We can observe a huge speed-up caused 
by the clearing optimization as already reported in~\cite{ck-persistent}.
We can see that our chunk algorithm performs slightly worse than the one of~\cite{ck-persistent} when executed
sequentially, but faster when parallelized. We also observe that standard,
twist, and chunk generally behave worse 
when computing persistent cohomology,
except for the Rips filtration of the Mumford data set, where the converse is true.

\begin{table}[htb]%

\begin{center}
\begin{tabular*}{0.99\textwidth}{@{\extracolsep{\fill}} l|r|r|r|r|r||r|r|r}
	\multicolumn{1}{l|}{Dataset}& \multicolumn{1}{c|}{$n \cdot 10^{-6}$} & \multicolumn{1}{c|}{std.~\cite{elz-topological}} &	\multicolumn{1}{c|}{twist~\cite{ck-persistent}} & \multicolumn{1}{c|}{cohom.~\cite{dmv-dualities}}  & \multicolumn{1}{c||}{DMT~\cite{Guenther_PLD2}}&	\multicolumn{1}{c|}{ $g / n$} &	\multicolumn{1}{c|}{ chunk (1x)} &	\multicolumn{1}{c}{ chunk (12x) }					\\
	\hline%
	Smooth							&	16.6	&	383s			&	3.1s			  &		65.8s		&		2.0s	&	0\%				&	5.0s 				&	0.9s\\
	Smooth$^{\perp}$		&	16.6	& 432s			&	11.3s			  &		20.8s		&			--	&	0\%				&	6.3s 				&	0.9s\\
	\hline%
	Noise								&	16.6	&	336s			&	17.2s				&		15971s	&		13.0s	&	9\%			&	28.3s 			&	6.3s\\
	Noise$^{\perp}$			&	16.6	&	1200s			&	29.0s				&		190.1s	&			--	&	9\%				&	31.1s 			&	5.8s\\
	\hline%
	Mixed	  						&	16.6	&	330s 			&	5.8s				&		50927s	&		12.3s	&	5\%			&	21.6s	 			&	2.4s\\
	Mixed$^{\perp}$			&	16.6	&	446s  		&	13.0s				&		32.7s 	&			--	&	5\%			&	32.0s	 			&	2.9s\\
	\hline%
	Torus								&	0.6		&	52s  			&	0.3s				&		1.6s 		&			--	&	7\%			&	0.3s	 			&	0.1s\\
	Torus$^{\perp}$			&	0.6		&	24s  			&	0.3s				&		1.4s 		&			--	&	7\%			&	0.9s 				&	0.2s\\
	\hline%
	Mumford							&	2.4		&	38s  			&	35.2s				&		2.8s 		&			--	&	82\%			&	14.6s	 			&	1.8s\\
	Mumford$^{\perp}$		&	2.4		&	58s  			&	0.2s				&		184.1s	&			--	&	82\%			&	1.5s 				&	0.4s
	\end{tabular*}
\end{center}
\caption{Running time comparison of various persistent homology algorithms applied to the data sets described in Section~\ref{sec:experiments}. The last three columns contain information of the algorithm presented in this paper: the fraction of global columns $g/n$, and the running times using one and twelve cores, respectively.}
\label{tab:running_times}
\end{table}%

\section{Conclusion and Outlook}

We have presented an algorithm for persistent homology that includes two simple optimization techniques
into the reduction algorithm. It can be fully parallelized, except for the reduction 
of compressed global columns, whose number is often small. 
Besides our asymptotic complexity bounds, which give a detailed dependence on the parameters of the algorithm,
our experiments show that significant speed-ups can be achieved through parallelized persistence computation.
Similar observations have been made recently by Lewis and Zomorodian~\cite{lz-multicore}
for the computation of (non-persistent) homology; see also~\cite{lsv-spectral}.
We plan a more extensive discussion of the practical effects of our optimizations and parallelization in 
an extended version of this paper.

\paragraph{Acknowledgements}
The authors thank Chao Chen, Herbert Edelsbrunner, and Hubert Wagner for helpful discussions.

\end{document}